\documentclass[11pt]{article}
\usepackage[T1]{fontenc}
\usepackage{amsmath,amsthm,mathtools}
\usepackage{libertine}
\usepackage[libertine]{newtxmath}
\usepackage{microtype}
\usepackage[letterpaper,margin=1in]{geometry}
\usepackage[colorlinks=true,linkcolor=blue,citecolor=blue,urlcolor=blue]{hyperref}
\usepackage[capitalise,nameinlink,noabbrev]{cleveref}
\Crefformat{equation}{#2(#1)#3}

\newtheorem{theorem}{Theorem}
\newtheorem{lemma}[theorem]{Lemma}
\title{A Tight Fractional Version of Generalized Tuza's Conjecture}
\author{Sichen Wang\\ \small Shenzhen MSU-BIT University\\ \small\texttt{wsc@smbu.edu.cn}}
\date{}
\hypersetup{pdftitle={A Tight Fractional Version of Generalized Tuza's Conjecture},pdfauthor={Sichen Wang}}

\begin{document}
\maketitle
\begin{abstract}
For an $r$-uniform hypergraph $H$, let $\nu(H)$ be the maximum number of edges no two of which share $r-1$ vertices, and $\tau(H)$ the minimum number of $(r-1)$-sets such that every edge contains one of them. Aharoni and Zerbib conjectured that $\tau(H)\le\lceil\frac{r+1}{2}\rceil\,\nu(H)$, which for $r=3$ generalizes Tuza's conjecture on triangles. We prove that the fractional relaxation $\tau^*(H)$ of $\tau(H)$ satisfies $\tau^*(H)\le\frac{r+1}{2}\,\nu(H)$. This constant is best possible for every $r$, and for $r\ge4$ it improves on the previous bound of roughly $3r/4$.
\end{abstract}

Tuza~\cite{Tuz90} conjectured that the triangles of every graph can be met by at most twice as many edges as the maximum number of edge-disjoint triangles. Aharoni and Zerbib~\cite{AZ20} proposed a generalization to hypergraphs. For $r\ge2$, an \emph{$r$-graph} $H$ is a set of $r$-element subsets, called edges, of a finite vertex set. An \emph{$(r-1)$-matching} of $H$ is a set of edges no two of which share $r-1$ vertices. Let $\nu(H)$ be the maximum size of an $(r-1)$-matching, and let $\tau(H)$ be the minimum number of $(r-1)$-sets of vertices such that every edge contains one of them; Aharoni and Zerbib write $\nu^{(r-1)}(H)$ and $\tau^{(r-1)}(H)$. An $(r-1)$-set lies in at most one edge of an $(r-1)$-matching, and by maximality every edge contains an $(r-1)$-subset of an edge of a maximum $(r-1)$-matching. Hence $\nu(H)\le\tau(H)\le r\,\nu(H)$. Aharoni and Zerbib conjectured that
\[
  \tau(H)\le\Big\lceil\frac{r+1}{2}\Big\rceil\,\nu(H).
\]
For $r=3$, applying this to the vertex sets of the triangles of a graph gives Tuza's conjecture. The constant $\lceil\frac{r+1}{2}\rceil$ would be best possible: the complete $r$-graph $K^{(r)}_{r+1}$ on $r+1$ vertices has $\nu=1$ and $\tau=\lceil\frac{r+1}{2}\rceil$~\cite{AZ20}, since any two of its edges share an $(r-1)$-set and each $(r-1)$-set lies in at most two of them.

The fractional relaxation $\tau^*(H)$ of $\tau(H)$ is the minimum total weight of an assignment of nonnegative weights to the $(r-1)$-sets such that the $(r-1)$-subsets of every edge receive total weight at least~$1$. A \emph{fractional $(r-1)$-matching} of $H$ assigns a weight $y_f\ge0$ to every edge $f$ so that the edges containing any $(r-1)$-set have total weight at most~$1$. By linear programming duality, $\tau^*(H)$ is also the maximum total weight of a fractional $(r-1)$-matching, so $\nu(H)\le\tau^*(H)\le\tau(H)$. The conjecture therefore has two fractional relaxations~\cite{AZ20}, $\tau^*(H)\le\lceil\frac{r+1}{2}\rceil\,\nu(H)$ and $\tau(H)\le\lceil\frac{r+1}{2}\rceil\,\tau^*(H)$. Krivelevich~\cite{Kri95} proved both for triangles, and Aharoni and Zerbib~\cite{AZ20} proved both for $r=3$, the first as a special case of their bound $\tau^*(H)\le(r-1)\,\nu(H)$ for $r\ge3$, derived from a theorem of F\"uredi~\cite{Fur81}. For every $r$, Guruswami and Sandeep~\cite{GS25} proved the second up to lower-order terms, with the constant $\frac r2+O(\sqrt{r\log r})$. For the first, Basit, McGinnis, Simmons, Sinnwell, and Zerbib~\cite{BMS22} improved the constant $r-1$ to $8/3$ for $r=4$ and to slightly below $3r/4$ for $r\ge5$, and noted that $\tau^*(K^{(r)}_{r+1})=\frac{r+1}{2}$, so no smaller constant is possible. We close this gap.

\begin{theorem}\label{thm:main}
Every $r$-graph $H$ with $r\ge2$ satisfies
\[
  \tau^*(H)\le\frac{r+1}{2}\,\nu(H),
\]
with equality for $H=K^{(r)}_{r+1}$.
\end{theorem}

Thus the first relaxation holds for every $r$, with the conjectured constant for odd $r$ and a constant smaller by $\frac12$ for even $r$. \Cref{thm:main} is new for $r\ge4$.

To prove \Cref{thm:main}, let $y$ be a fractional $(r-1)$-matching of an $r$-graph $H$, let $\nu=\nu(H)$, and write $y(A)=\sum_{f\in A}y_f$ for $A\subseteq H$. By duality, it suffices to show that $y(H)\le\frac{r+1}{2}\nu$. Among all $(r-1)$-matchings of size $\nu$, choose $M$ maximizing $y(M)$, and let $\mathcal S$ be the set of $(r-1)$-subsets of edges of $M$, as in the bound $\tau\le r\nu$. Since no two edges of $M$ share $r-1$ vertices, each set in $\mathcal S$ lies in exactly one edge of $M$, and $|\mathcal S|=r\nu$. By maximality, every edge outside $M$ shares $r-1$ vertices with some edge of $M$. Let $P$ be the set of edges outside $M$ that share $r-1$ vertices with exactly one edge of $M$, and let $R$ be the set of the other edges outside $M$. For $e\in M$, let $P_e$ be the set of edges of $P$ that share $r-1$ vertices with $e$, so that $P$ is the disjoint union of the sets $P_e$. Every edge of $M$ contains $r$ sets of $\mathcal S$, every edge of $P$ at least one, and every edge of $R$ at least two, lying in different edges of $M$. Summing the constraints of $y$ at the sets in $\mathcal S$ gives
\begin{equation}\label{eq:count}
  r\,y(M)+y(P)+2\,y(R)\le r\nu.
\end{equation}
The edges of $P$ are counted only once in \Cref{eq:count}, so it remains to bound $y(P)$.

\begin{lemma}\label{lem:local}
Every $e\in M$ satisfies $y(P_e)\le1+(r-2)\,y_e$.
\end{lemma}

\begin{proof}
Each $f\in P_e$ shares $r-1$ vertices with no edge of $M\setminus\{e\}$, so $(M\setminus\{e\})\cup\{f\}$ is an $(r-1)$-matching of size $\nu$, and $y_f\le y_e$ by the choice of $M$. Similarly, any two edges of $P_e$ share $r-1$ vertices, as otherwise replacing $e$ by both would give an $(r-1)$-matching of size $\nu+1$. Write $e-u+v$ for $(e\setminus\{u\})\cup\{v\}$; every edge of $P_e$ has this form with $u\in e$ and $v\notin e$. If $P_e=\emptyset$, there is nothing to prove; otherwise fix $e-u_0+v_0\in P_e$. The edges $e-u_0+v$ of $P_e$ contain $e\setminus\{u_0\}$, as does $e$, so their total weight is at most $1-y_e$. Every other edge $e-u+v$ of $P_e$ has $u\ne u_0$, so it meets $e-u_0+v_0$ in the $r-2$ vertices of $e\setminus\{u,u_0\}$, and in $v$ if $v=v_0$. As any two edges of $P_e$ share $r-1$ vertices, $v=v_0$, so the other edges have the form $e-u+v_0$ with $u\in e\setminus\{u_0\}$, and each has weight at most $y_e$. Therefore $y(P_e)\le(1-y_e)+(r-1)\,y_e=1+(r-2)\,y_e$.
\end{proof}

Maximizing $y(M)$ is used only in \Cref{lem:local}, and it is needed there: in $K^{(r)}_{r+1}$ with $r\ge3$, let $y$ be $0$ on an edge $e$ and $\frac12$ elsewhere; then $M=\{e\}$ is a maximum $(r-1)$-matching, but $y(P_e)=\frac r2>1+(r-2)\,y_e$.
\begin{proof}[Proof of \Cref{thm:main}]
Summing \Cref{lem:local} over $e\in M$ gives $y(P)\le\nu+(r-2)\,y(M)$, and adding this to \Cref{eq:count} yields $2\,y(H)=2\,y(M)+2\,y(P)+2\,y(R)\le(r+1)\nu$. For $H=K^{(r)}_{r+1}$, the weight $\frac12$ on each of the $r+1$ edges is a fractional $(r-1)$-matching, as each $(r-1)$-set lies in at most two edges; since $\nu=1$, equality holds.
\end{proof}

\end{document}